\documentclass{amsart}
\usepackage{amsthm}
\usepackage{amsmath}
\usepackage{amssymb}
\usepackage[all]{xy}
\usepackage{enumerate}

\theoremstyle{plain}
\newtheorem{thm}{Theorem}[section]
\newtheorem{lemma}[thm]{Lemma}
\newtheorem{cor}[thm]{Corollary}
\newtheorem{prop}[thm]{Proposition}

\theoremstyle{definition}
\newtheorem{defn}[thm]{Definition}
\newtheorem{remark}[thm]{Remark}
\newtheorem{ex}[thm]{Example}

\begin{document}

\title{Module Invariants and Blocks of Finite Group Schemes}

\author[ Paul Sobaje]
{Paul Sobaje}

\begin{abstract}

\sloppy{
We investigate various topological spaces and varieties which can be associated to a block of a finite group scheme $G$.  These spaces come from the theory of cohomological support varieties for modules, as well as from the representation-theoretic constructions of E. Friedlander and J. Pevtsova.
}

\end{abstract}


\maketitle

Let $G$ be a finite group scheme over an algebraically closed field $k$ of characteristic $p > 0$, and let $k[G]$ denote the coordinate algebra (or representing algebra) of $G$.  The algebra $k[G]$ is a finite dimensional commutative Hopf algebra, and representations of $G$ are equivalent to right co-modules for $k[G]$, which in turn are equivalent to left modules of the ``group algebra" $kG := \text{Hom}_k(k[G],k)$.  As $kG$ is finite dimensional, it can be decomposed uniquely as an algebra into the direct product of its indecomposable two-sided ideals called the blocks of $kG$.  Any $kG$-module breaks up as the direct sum of modules which lie in a block, reducing the study of $kG$-mod to the study of the module categories of its blocks.

In the case that $G$ is a finite group, the representation theory of a block is governed to a certain extent by its defect group, which is a particular $p$-subgroup of $G$ unique up to conjugacy.  Thus it would seem desirable to study the blocks of a finite group scheme by adapting the theory of defect groups to this more general setting, with unipotent subgroup schemes filling the role of $p$-subgroups.  However, key features of $p$-subgroups not shared by arbitrary unipotent subgroup schemes seem to stand in the way of making such a generalization.  Another obstacle is presented by the work of R. Farnsteiner and A. Skowro\'{n}ski in \cite{FSk}, where they prove that for an arbitrary finite group scheme, the principal block is not always the block with the most complicated representation type.  This is a divergence from the well-known situation for group algebras of finite groups, and gives some indication that extending the theory of defect groups to all finite group schemes either might not be possible, or at the very least will behave rather differently.

Even without a generalized defect theory, developments over the last several years in the theory of support varieties for modules of finite group schemes has led to new means by which blocks can be studied.  Fundamental to the considerations in this paper is the work of E. Friedlander and J. Pevtsova in \cite{FP1}, in which they defined for a finite group scheme $G$ the ``representation-theoretic" topological space of $p$-points $P(G)$, and proved that this space is homeomorphic to the projective variety of the cohomology ring $\text{H}^{\bullet}(G,k)$.  Their work not only extended to modules of finite group schemes many of the properties of support varieties known to be true for modules of finite groups, but the creation of $p$-points has also led to new and interesting module invariants (see, for example, \cite{CFP}, \cite{FP2}, \cite{FP4}, and \cite{FPS}).  As it relates to blocks, one important application of \cite{FP1} was given by Farnsteiner, who proved in \cite{Far2} that support varieties can be used to determine if a block has wild representation type.  Specifically, let $V_G(\mathcal{B})$ denote the union of the support varieties of all simple modules lying in a block $B$.  Farnsteiner proved that if $\text{dim} V_G(\mathcal{B}) \ge 3$, then $B$ has wild representation type.

In this paper we make some investigations into various support spaces of blocks.  After quickly recalling relevant background material and setting our notation in the first section, in Section 2 we consider for a block $B$ the variety $V_G(\mathcal{B})$ as defined above.  We show in Theorem \ref{same} that by taking $B$ to be a $kG$-module under the ``left-adjoint" action, there is an equality of varieties $V_G(\mathcal{B}) = V_G(B)$.  We also demonstrate that there is an indecomposable summand $M$ of $B$ such that $V_G(B) = V_G(M)$, and deduce using a theorem of J. Carlson \cite{Car} that the projective variety $\text{Proj }V_G(\mathcal{B})$ is connected.

In Section 3, we look at how $V_G(\mathcal{B})$ compares with the variety $X_B$ defined using the Hochschild cohomology ring of $B$.  This question is motivated by work done for finite groups by M. Linckelmann in \cite{L1} and \cite{L2}.  We prove in Proposition \ref{krull} that there is a finite surjective morphism of varieties $X_B \rightarrow V_G(\mathcal{B})$.  This result follows as an easy consequence of the results in the previous section.  We prove that for the principal block $B_0$ of $kG$, there is an isomorphism of varieties $X_{B_0} \cong V_G(\mathcal{B}_0)$, provided that either $B_0$ is local, or that the complexity of the trivial module is $\le 1$.

The motivation behind the comparison of $X_B$ and $V_G(\mathcal{B})$ is that $X_B$ is defined only in terms of $B$ and hence is a true invariant of $B$, as opposed to $V_G(\mathcal{B})$ whose definition involves the group algebra in which $B$ is a summand.  In Section 4, we look at this question from the perspective of $p$-points.  We set $P(G)_{\mathcal{B}} = \bigcup P(G)_{S_i}$, for all simple modules lying in $B$, and compare this to the space of \textit{flat-points} of $B$, denoted $F(B)$, which is defined by taking flat maps from $k[t]/(t^p)$ to $B$ (our notation and definition are slightly modified from the definition of a flat-point given in \cite{Far1}).  The projection map from $kG$ onto the block $B$ defines a map from $P(G)_{\mathcal{B}}$ to $F(B)$, which we show in Proposition \ref{injective} to be injective.  We also show, as a sort of analogue to a result in Section 3, that if $B_0$ is the principal block of $kG$, and if the trivial module is the only simple $B_0$-module, then there is a homeomorphism $P(G)_{\mathcal{B}_0} \cong F(B_0)$.  The main step in proving this comes in Theorem \ref{equiv}, which shows that if $G$ is a unipotent finite group scheme, then every flat map $k[t]/(t^p) \rightarrow kG$ is equivalent to one which factors through an abelian subgroup scheme (under the equivalence defined on such maps in \cite{FP1}).  This implies that, in terms of providing a representation-theoretic topological space which is homeomorphic to $\text{Proj H}^{\bullet}(G,k)$, the definition of a $p$-point could effectively drop the word ``abelian" from its factorization requirement.  We note however that we are unaware at this point how such an alteration in definition would affect other theories coming from $p$-points, such as modules of ``constant Jordan type" (see \cite{CFP}).

\subsection{Acknowledgements}
We wish to thank Eric Friedlander, without whom the writing of this paper would not be possible.  In particular his constant support and considerable mathematical insight were invaluable to our investigations.  We are also indebted to Julia Pevtsova for numerous helpful conversations and observations, and of course for developing, in collaboration with Eric, the tools used in the analysis of this paper.  The interest in blocks of finite group schemes was motivated by the work of Rolf Farnsteiner, and the results of Section 2 were aided by many helpful conversations with Rolf.  We would also like to thank Sarah Witherspoon for clarifying conversations about Hochschild cohomology for Hopf algebras.  Finally, we thank the referee for many helpful comments and observations.

\section{Notation and Recollections}

We will assume throughout that $k$ is an algebraically closed field of characteristic $p>0$.  Unless specified, tensor products are assumed to be over $k$.  If $G$ is a finite group scheme over $k$, we write its coordinate ring as $k[G]$, and define $kG := \text{Hom}_k(k[G],k)$.  Following the terminology of \cite{FP1}, we call $kG$ the \textit{group algebra} of $G$.  It is a finite dimensional cocommutative Hopf algebra with comultiplication $\Delta$, counit $\epsilon$, and antipode $s$.  The category $kG$-mod of finitely generated left $kG$-modules is equivalent to the category of finitely generated representations of $G$, and thus we will speak of the two interchangeably.  As an algebra, $kG$ is a direct product of its indecomposable two-sided ideals, which we call the blocks of $kG$, and write as $kG = B_0 + \cdots + B_r$.  By $e_i$ we denote the central idempotent of $B_i$, so that $B_i = kGe_i$.  A $kG$-module $M$ is said to lie in the block $B_i$ if $e_i$ acts as the identity map on $M$.  In the above decomposition of $kG$, the block $B_0$ will always denote the principal block of $kG$; it is the block in which the trivial module $k$ lies.

For a $kG$-module $M$, the cohomology groups $\text{H}^i(G,M)$ are defined to be the groups $\text{Ext}^i_{kG}(k,M)$.  We set 

$$\text{H}^{\bullet}(G,k) = \begin{cases} \bigoplus_{i \ge 0} \text{H}^i(G,k) & \text{if char $k = 2$} \\  \bigoplus_{i \ge 0} \text{H}^{2i}(G,k) & \text{if char $k \ne 2$} \end{cases}$$

\bigskip
\noindent It is a finitely generated commutative algebra over $k$ (\cite[1.1]{FS}), and and we denote by $V_G$ the maximal ideal spectrum of $\text{H}^{\bullet}(G,k)$.

\bigskip
Following the notation and terminology in \cite{Ben2}, for $M,N \in kG$-mod, $I_G(M,N)$ denotes the annihilator in $\text{H}^{\bullet}(G,k)$ of $\text{Ext}^*(M,N)$ as a module given by the cup product.  Equivalently, this is the annihilator of the module $\text{H}^*(G,\text{Hom}_k(M,N))$.  The relative support variety $V_G(M,N)$ is then the set of maximal ideals in $V_G$ which contain $I_G(M,N)$.

If $N=M$, we simply write $I_G(M)$ and $V_G(M)$, and call $V_G(M)$ the support variety of $M$.  Note that in this case, $I_G(M)$ can be given as the kernel of the map of graded algebras from $\text{H}^{\bullet}(G,k)$ to $\text{H}^{\bullet}(G,\text{Hom}_k(M,M))$, this map induced by the inclusion $k \hookrightarrow \text{Hom}_k(M,M)$.  See \cite[1.5, 5.6]{FP1} for a list of properties satisfied by support varieties for modules of finite group schemes.

\bigskip
For an algebra $A$ we denote by $\text{HH}^i(A)$ the $i$-th Hochschild cohomology group of $A$ with coefficients in $A$, and define the group by $\text{HH}^i(A) := \text{Ext}^i_{A \otimes A^{op}}(A,A)$, where $A$ is a left $A \otimes A^{op}$-module in the usual way.  If $\zeta_1 \in \text{HH}^n(A)$, and $\zeta_2 \in \text{HH}^m(A)$, then by regarding these as $n$-fold and $m$-fold extensions of $A$ by $A$ respectively, we can tensor over $A$ to obtain an $(n+m)$-fold extension $\zeta_1 \smile \zeta_2 \in \text{HH}^{n+m}(A)$, and this gives the space $\text{HH}^*(A)$ the structure of an associative algebra, which was shown to be graded-commutative by M. Gerstenhaber in \cite{G}.  Just as with $\text{H}^{\bullet}(G,k)$, we denote by $\text{HH}^{\bullet}(A)$ the even Hochschild cohomology ring.

We will later make use of the well-known fact that the decomposition of $kG$ into blocks yields an algebra decomposition $\text{HH}^{\bullet}(kG) \cong  \text{HH}^{\bullet}(B_0) \times \cdots \times \text{HH}^{\bullet}(B_r)$.

\bigskip
A $p$-point \cite{FPE} of a finite group scheme $G$ is a map of algebras $\alpha: k[t]/(t^p) \rightarrow kG$, such that:

\begin{enumerate}

\item $\alpha^*(kG)$ is a projective $k[t]/(t^p)$-module.
\item $\alpha$ factors through a unipotent, abelian subgroup scheme of $G$.

\end{enumerate}

Two $p$-points $\alpha$ and $\beta$ are equivalent, written $\alpha \sim \beta$, if $\alpha^*(M) \text{ projective } \iff \beta^*(M) \text{ projective}$, for all $M$ in $kG$-mod.

By $P(G)$ we denote the set of all equivalence classes of $p$-points of G.  For a $kG$-module $M$, $P(G)_M$ is the set $\{ [\alpha] \in P(G) | \, \alpha^*(M) \text{ is not projective} \}$.  Declaring the closed sets of $P(G)$ to be all of the sets $P(G)_M, M \in kG$-mod, defines  a Noetherian topology on $P(G)$ (\cite[3.10]{FP1}), and the space $P(G)$ with this topology is called the space of $p$-points of $G$.  This space provides a non-cohomological description of support varieties, as shown in the following theorem.

\begin{thm}\label{p-points}\cite[4.11]{FP1}
There is a homeomorphism 

\vspace{0.03 in}
\begin{center} $\Psi: P(G) \stackrel{\sim}{\rightarrow} \textup{Proj H}^{\bullet}(G,k)$ \end{center}
\vspace{0.03 in}

\noindent satisfying the property that

\vspace{0.03 in}
\begin{center} $\Psi^{-1}(\textup{Proj }V_G(M)) = P(G)_M.$ \end{center}
\vspace{0.03 in}

\noindent for every finitely generated $G$-module $M$.
\end{thm}

\section{The Variety $V_G(\mathcal{B})$}

Let $G$ be a finite group scheme with group algebra $kG$, and let $B$ be a block of $kG$.  There is a smallest closed subset of $V_G$ which contains $V_G(M)$ for all $M$ lying in $B$, which we will denote by $V_G(\mathcal{B})$.  Basic properties of support varieties can be used to show that if $\{ S_i \}$ is a complete set of non-isomorphic simple $B$-modules, then

$$V_G(\mathcal{B}) = \bigcup_i V_G(S_i)$$

\bigskip
Similarly, we can define $P(G)_{\mathcal{B}} := \bigcup P(G)_{M}$, the union being over all modules $M$ lying in $B$.  By Theorem \ref{p-points}, $P(G)_{\mathcal{B}} \cong \text{Proj }V_G(\mathcal{B})$.

\bigskip
The usefulness of these spaces can be seen in this next theorem due to Farnsteiner.

\begin{thm} \cite[3.1]{Far2}
Let $B$ be a block of a finite group scheme $G$.

\begin{enumerate}

\item dim $V_G(\mathcal{B}) = 0$ if and only if $B$ is a simple algebra.
\item If dim $V_G(\mathcal{B}) \ge 2$, then $B$ has infinite representation type.
\item If dim $V_G(\mathcal{B}) \ge 3$, then $B$ has wild representation type.

\end{enumerate}

\end{thm}

\begin{remark}
The theorem referenced from \cite{Far2} proves (3), but we have included the other two cases for completeness.  Part (1) is immediate given the properties of support varieties, and (2) follows from a result of Heller \cite{Hel} for self-injective algebras having finite representation type.
\end{remark}

\bigskip
The use of $\mathcal{B}$ in our notation above is to indicate that we are not considering the support variety of the module $B$ as a summand of the left regular representation, as this variety is always just a single point.  We can however consider $B$ as a module under the left-adjoint action of $kG$.  Recall that for $x \in kG, b \in B$, and writing $\Delta(x) = \sum x_{(1)} \otimes x_{(2)}$, the left-adjoint action is given by

$$x.b = \sum x_{(1)}bs(x_{(2)})$$

With this action, $B$ is a $G$-algebra.  That is, the multiplication map $m: B \otimes B \rightarrow B$ is a map of $kG$-modules.  From this point on, any reference to $B$ as a $kG$-module will assume it is given by the left-adjoint action.  The module $B$ does not in general lie in the block $B$ (for instance the trivial module is a composition factor), however as our next result shows, $V_G(B)$ is equal to $V_G(\mathcal{B})$.

We first establish a few lemmas in order to simplify the proof.  Recall that for a module $M$, the fixed points of $M$, denoted $M^G$, are those $m \in M$ such that $x.m = \epsilon(x)m$.

\begin{lemma}\label{center}
The fixed points of $B$ are precisely the elements in the center of $B$.
\end{lemma}

\begin{proof}
If $b \in Z(B)$, then $x.b = \sum x_{(1)}bs(x_{(2)}) = \sum x_{(1)}s(x_{(2)})b = \epsilon(x)b$.  Conversely, if $y.b = \epsilon(y)b$ for all $y \in kG$, then we have $xb = \sum x_{(1)}\epsilon(x_{(2)})b = \sum x_{(1)}b\epsilon(x_{(2)})$, which by coassociativity and the definition of the antipode we can express as $\sum x_{(1)}bs(x_{(2)})x_{(3)}$.  In view of the fact that $b \in B^G$, this is then equal to $\sum \epsilon(x_{(1)})bx_{(2)} = \sum b\epsilon(x_{(1)})x_{(2)} = bx$.
\end{proof}

\begin{lemma}\label{relative}
Let $A$ be a finite dimensional $G$-algebra.  Then $I_G(A) = I_G(k,A)$; therefore $V_G(A) = V_G(k,A)$
\end{lemma}

\begin{proof}
The action of $\text{H}^{\bullet}(G,k)$ on $\text{Ext}^*_{kG}(A,A)$ via the cup product factors as the algebra map $\text{H}^{\bullet}(G,k) \rightarrow \text{H}^*(G,A)$ followed by the action $\text{H}^*(G,A) \otimes \text{Ext}^*_{kG}(A,A) \rightarrow \text{Ext}^*_{kG}(A,A \otimes A) \rightarrow \text{Ext}^*_{kG}(A,A)$, the last map induced by the multiplication map of $A$.  Thus $I_G(k,A) \subseteq I_G(A)$.  On the other hand, as shown in \cite[5.7]{Ben2} (and the same arguments holding for finite group schemes), the inclusion $I_G(A) \subseteq I_G(k,A)$ follows from the cup product of $\text{H}^{\bullet}(G,k)$ on $\text{H}^*(G,A)$ factoring through $\text{Ext}^*_{kG}(A,A)$ acting via the Yoneda product on $\text{H}^*(G,A)$.
\end{proof}

\begin{thm}\label{same}
Let $B$ be a block of $kG$.  With the left-adjoint action of $kG$ on $B$, we have

\begin{enumerate}

\item $V_G(B) = V_G(\mathcal{B})$.
\item $V_G(B) = V_G(M)$, for $M$ some indecomposable summand of $B$.  In particular, the projective variety $\text{Proj } V_G(\mathcal{B})$ is connected.

\end{enumerate}

\end{thm}

\begin{proof}
If $M$ is any module lying in $B$, then the map of $G$-algebras $k \rightarrow \text{Hom}_k(M,M)$ factors as $k \rightarrow B \rightarrow \text{Hom}_k(M,M)$.  Thus the map $\text{H}^{\bullet}(G,k) \rightarrow \text{H}^{\bullet}(G,\text{Hom}_k(M,M))$ factors through the map $\text{H}^{\bullet}(G,k) \rightarrow \text{H}^{\bullet}(G,B)$.  We then have $I_G(k,B) \subseteq I_G(M)$ for all $M \in B$-mod, so that $V_G(M) \subseteq V_G(k,B)$ for all $M$ in $B$-mod.  Hence $V_G(\mathcal{B}) \subseteq V_G(k,B)$.

Conversely, suppose that $\zeta \in I_G(M)$ for all $M$ in $B$.  In particular $\zeta \in I_G(S)$ for all simple $B$-modules.  The Jacobson radical of $B$, $J(B)$, is a submodule of $B$, and there is an isomorphism of $G$-algebras $B/J(B) \cong \bigoplus \text{Hom}_k(S_i,S_i)$, for $\{S_i\}$ a set of non-isomorphic simple $B$-modules.  In the composite of maps

$$\text{H}^{\bullet}(G,k) \rightarrow \text{H}^{\bullet}(G,B) \rightarrow \text{H}^{\bullet}(G,B/J(B))$$

\noindent it follows then that $\zeta$ is sent to $0 \in \text{H}^{\bullet}(G,B/J(B))$.  However, from the short exact sequence of modules

$$0 \rightarrow J(B) \rightarrow B \rightarrow B/J(B) \rightarrow 0$$

\noindent we get a long exact sequence in cohomology, which in particular tells us that the kernel of the map $\text{H}^{\bullet}(G,B) \rightarrow \text{H}^{\bullet}(G,B/J(B))$ is given by the image of the map $\text{H}^{\bullet}(G,J(B)) \rightarrow \text{H}^{\bullet}(G,B)$.  As $J(B)$ is a nilpotent ideal, this image is a nilpotent ideal: if $\gamma \in  \text{H}^{2i}(G,B)$ can be represented by a map $\Omega^{2i}(k) \rightarrow B$ whose image is contained in $J(B)$, then by the definition of the cup product, $\gamma^n$ can be represented by map $\Omega^{2in}(k) \rightarrow B$ whose image is contained in the image of the map $J(B)^{\otimes n} \rightarrow J(B)$, $b_1 \otimes \cdots \otimes b_n \mapsto b_1 \cdots b_n$.  For large enough $n$ this image is $0$, and hence $\gamma$ is nilpotent.

Thus, the image of the element $\zeta$ in the map $\text{H}^{\bullet}(G,k) \rightarrow \text{H}^{\bullet}(G,B)$ is nilpotent, so that $\zeta$ is in the radical of $I_G(k,B)$.  Thus $V_G(k,B) \subseteq \bigcup V_G(S) = V_G(\mathcal{B})$.  Applying lemma \ref{relative}, we get $V_G(B) = V_G(\mathcal{B})$.

For the proof of (2), let $B \cong M_1 + \cdots + M_n$ be a direct sum decomposition into indecomposable submodules.  The space of fixed points of $B$ is equal to the sum of fixed spaces $M_1^G + \cdots + M_n^G$.  By lemma \ref{center}, $B^G = Z(B)$, which is a local algebra, and so at least one $M_i$ must contain the element $e + z$, where $e$ is the central idempotent of $B$ and $z$ is central and nilpotent.  For $N$ any $B$-module, the map $f: M_i \otimes N \rightarrow N$ given by $f(m \otimes n) = mn$ is a well-defined map coming from the action of $B$ on $N$.  It is a $G$-module map, since if $x \in kG$, then 

\vspace{-0.1 in}
$$f(x(m \otimes n)) =  f(\sum x_{(1)}ms(x_{(2)}) \otimes x_{(3)}n) = \sum x_{(1)}ms(x_{(2)})x_{(3)}n = xmn = xf(m \otimes n)$$

There is also a map $h: N \rightarrow M_i \otimes N$ given by $h(n) = (e+z) \otimes \sum_{i \ge 0} (-z)^in$, which is well-defined because $z$ is nilpotent.  Since $e,z$ are central in $B$, then for $x \in kG$ we have 

\vspace{-0.1 in}
$$(e+z) \otimes (\sum_{i \ge 0} (-z)^ixn) = \sum\left( \epsilon(x_{(1)})(e+z) \otimes x_{(2)}(\sum_{i \ge 0} (-z)^in) \right)$$

The right-hand sum is equal to $x.((e+z) \otimes (\sum_{i \ge 0} (-z)^in))$, so that $h$ is a $G$-module map.  Since $f \circ h = id_N$, $N$ is a summand of $M_i \otimes N$.  It follows that $V_G(N) \subseteq V_G(M_i)$ for all modules lying in $B$, and hence $V_G(\mathcal{B}) \subseteq V_G(M_i)$.  By part (1), and the fact that $M_i$ is a direct summand of $B$, this subset inclusion is actually an equality.

The connectedness of the variety $\text{Proj }V_G(M_i)$ is given by Carlson's theorem \cite{Car} on indecomposable modules for finite groups, the proof of which applies to the setting of finite group schemes. 
\end{proof}

\begin{remark}
As pointed out to us by J. Pevtsova, part (1) in the above theorem is analogous to a result of A. Premet in the context of reduced enveloping algebras (\cite[2.2]{Pr}), although the two methods of proof are quite different.
\end{remark}

We conclude this section by recording the relationship between the support variety of a block of a finite group, and the support coming from a defect group of the block.  We state the result for $p$-points, which simplifies one aspect of the proof.  We note that if $G$ is any finite group scheme with closed subgroup scheme $H$, then the inclusion $H \subseteq G$ induces a natural map $i: P(H) \rightarrow P(G)$.

\begin{prop}
Let $G$ be a finite group, $B$ a block of $kG$ with defect group $D$, and let $i: P(D) \rightarrow P(G)$ be the natural map on $p$-support spaces.  Then 
$$P(G)_{\mathcal{B}} = i(P(D)).$$
\end{prop}

\begin{proof}
Since every module lying in $B$ is the summand of a module induced from $kD$ to $kG$, we have by \cite[4.12]{FPS} that $P(G)_M \subseteq i(P(D))$ for all modules $M$ lying in $B$, and hence $P(G)_{\mathcal{B}} \subseteq i(P(D))$.  On the other hand, there is some module $M^{\prime}$ lying in $B$ which is a trivial source module and has vertex $D$ \cite[6.3.3]{Ben1}.  It follows that the trivial module is a summand of $M^{\prime}$ when restricted to $kD$.  If $\alpha$ is a $p$-point factoring through $kD$, then in view of the previous statement, we have $[\alpha] \in P(G)_{M^{\prime}} \subseteq P(G)_{\mathcal{B}}$.  Thus, $i(P(D)) \subseteq P(G)_{\mathcal{B}}$.
\end{proof}

\section{Comparison with Hochschild Cohomology}

Let $X_B$ be the maximal ideal spectrum of $\text{HH}^{\bullet}(B)$.  We will show that as an easy consequence of the results in the previous section, there is a finite surjective morphism of varieties from $X_B$ to $V_G(\mathcal{B})$, and then proceed to show a few instances in which the two are isomorphic. But first we will recall how the cohomology ring of a finite group scheme and its Hochschild cohomology ring relate to each other, citing as we go the appendix of \cite{PW}, which works these details out nicely for general finite dimensional Hopf algebras.

Following \cite{PW}, denote by $\delta$ the composite of the maps $(\text{Id} \otimes s) \circ \Delta$.  Then by \cite[7.1, 7.2]{PW}, $\delta$ defines an embedding of $kG$ into $kG \otimes kG^{op}$ such that

\begin{enumerate}

\item $kG \otimes kG^{op}$ is a projective $kG$-module, $kG$ acting via $\delta$

\item As $kG \otimes kG^{op}$-modules, $kG \cong (kG \otimes kG^{op}) \otimes_{\delta(kG)} k$

\end{enumerate}

At the same time, the left-adjoint action of $kG$ on itself is clearly just the restriction (via the embedding $\delta$) of the natural action of $kG \otimes kG^{op}$ on $kG$.  Thus, we can apply the Eckmann-Shaprio isomorphism to get
$$\text{H}^i(G,kG) = \text{Ext}^i_{kG}(k,kG) \cong \text{Ext}^i_{kG \otimes kG^{op}}(kG,kG) = \text{HH}^i(kG)$$

This proves that there is an isomorphism of vector spaces $\text{H}^*(G,kG) \cong \text{HH}^*(kG)$, and as proven in \cite[7.2]{PW}, this is also an isomorphism of algebras.  It is then straight-forward to see that if $B$ is a block of $kG$, it is a summand of $kG$ as a $kG \otimes kG^{op}$-module, and we have $\text{Ext}^i_{kG}(k,B) \cong \text{Ext}^i_{kG \otimes kG^{op}}(kG,B) \cong \text{Ext}^i_{B \otimes B^{op}}(B,B)$, and consequently, $\text{H}^{\bullet}(G,B) \cong \text{HH}^{\bullet}(B)$.

With this isomorphism of algebras established, we have the following proposition.

\begin{prop}\label{krull}
Let $B$ be a block of a finite group scheme $G$.  Then there is a finite surjective morphism of varieties $X_B \rightarrow V_G(\mathcal{B})$.  In particular, the Krull dimension of $\textup{HH}^{\bullet}(B)$ is equal to the dimension of the variety $V_G(\mathcal{B})$.
\end{prop}

\begin{proof}
As just recalled, $\text{HH}^{\bullet}(B) \cong \text{H}^{\bullet}(G,B)$, and the latter is a finite module over $\text{H}^{\bullet}(G,k)$ by a theorem of Friedlander and Suslin (\cite[1.1]{FS}).  Thus, there is a finite surjective morphism of varieties $X_B \rightarrow V_G(k,B)$.  By Lemma \ref{relative} and Theorem \ref{same}, $V_G(k,B) = V_G(\mathcal{B})$.
\end{proof}

We next show that for certain principal blocks, there is an isomorphism, modulo nilpotent elements, between $\text{HH}^{\bullet}(B_0)$ and $\text{H}^{\bullet}(G,k)$.  For principal blocks of finite groups, M. Linckelmann was able to prove in \cite{L2} a much stronger result, proving that such an isomorphism holds in general.  We also note that S. Siegel and S. Witherspoon had previously proved in \cite{SW} this result for finite $p$-groups (in which case the group algebra is the principal block), and their proof is essentially the same as will be given to prove the first part of the following theorem.

\begin{thm}\label{nilpotents}
Let $B_0$ be the principal block of a finite group scheme $G$.  The rings $\textup{HH}^{\bullet}(B_0)$ and $\textup{H}^{\bullet}(G,k)$ are isomorphic modulo nilpotent elements in the following cases:

\begin{enumerate}

\item If the trivial module is the only simple $B_0$-module. 
\item If dim $V_G \le 1$.

\end{enumerate}

\end{thm}

\begin{proof}
For the principal block $B_0$ of any finite group scheme $G$, we have $\text{H}^{\bullet}(G,B_0) \cong \text{H}^{\bullet}(G,k) \oplus \text{H}^{\bullet}(G,I)$, where $k$ is the 1-dimensional vector space of $B_0$ spanned by the central idempotent, and $I$ is the augmentation ideal of $B_0$.  In the above decomposition, $\text{H}^{\bullet}(G,k)$ is a subalgebra of $\text{H}^{\bullet}(G,B_0)$ and $\text{H}^{\bullet}(G,I)$ is an ideal.  Thus we are ultimately interested in proving that $\text{H}^{\bullet}(G,I)$ is nilpotent.

Under the assumption that the trivial module is the only simple $B_0$-module, we have that $I = \text{Rad}(B_0)$, and is therefore a nilpotent ideal.  Suppose now that $\zeta \in \text{H}^{\bullet}(G,I)$.  Then $\zeta^n$ is in the image of the map $H^{\bullet}(G,I^{\otimes n}) \rightarrow \text{H}^{\bullet}(G,I)$, induced by multiplication in $I$, and by the nilpotence of $I$ this map is $0$ for large enough $n$.

For the proof of (2), if dim $V_G = 0$, then this case is trivial since $B_0 \cong k$, hence both $\text{H}^{\bullet}(G,B_0)$ and $\text{HH}^{\bullet}(B_0)$ are isomorphic to $k$.  If $V_G = 1$, then it follows from \cite[5.10.4]{Ben2} (again the proof applying also to finite group schemes) that $\Omega^n(k) \cong k$ for some $n$.  Let $\zeta \in \text{H}^{2i}(G,I)$.  The element $\zeta^n$ is represented by a map from $\Omega^{2in}(k) \cong k$ to $I$.  Thus, $\zeta^n$ is represented by a map whose image lands both in the augmentation ideal $I$, and in the center of $B_0$ by lemma \ref{center}.  But the center of $B_0$ is local, and hence its intersection with the augmentation ideal is precisely the set of nilpotent central elements of $B_0$.  It follows that $\zeta^{nm} = 0$ for large enough $m$, and hence the ideal is nilpotent.
\end{proof}

\begin{cor}
If $B_0$ is the principal block of $kG$, and satisfies either of the hypotheses in the previous theorem, then $X_{B_0} \cong V_G$.
\end{cor}

\bigskip
We now clarify the ramifications of a principal block having only one simple module, or equivalently, being a local $k$-algebra, as this hypothesis will again be used in the following section.  First we recall that a finite group scheme $H$ is called \textit{linearly reductive} if the group algebra $kH$ is semisimple.  We observe that if $U$ is a unipotent finite group scheme, then the principal block of $k(U \times H)$ is isomorphic as an algebra to $kU$, and thus is local.  It has been pointed out to us by the referee that the general case does not stray far from this example, as any $G$ which has a local principal block must be an extension of a unipotent finite group scheme by a linearly reductive group scheme.  The following proposition provides a proof of this claim, though it can also be deduced by applying results found in \cite{Far3} and \cite{FV} (as first observed by the referee).

\begin{prop}\label{localunipotent}
Let $B_0$ be the principal block of the finite group scheme $G$, and suppose that $B_0$ is local.  Then there is a linearly reductive normal subgroup scheme $N$ of $G$ such that $G/N$ is unipotent, and the canonical map $kG \rightarrow k(G/N)$ induces an algebra isomorphism $B_0 \cong k(G/N)$.
\end{prop} 

\begin{proof}
If the trivial module is the only simple module lying in the principal block, then $e_0 \otimes e_0$ is the unique non-zero idempotent of $B_0 \otimes B_0$.  Suppose now that $kG \cong B_0 + \cdots + B_r$, and let $B_k, \; k \neq 0$, be a block with central idempotent $e_k$.  We know that $\Delta(e_k)$ is an idempotent in $kG \otimes kG$, an algebra whose blocks are of the form $B_i \otimes B_j$, and so we must have

$$\Delta(e_k) = \sum_{0 \leq i,j \leq r} x_{ij}, \; \quad x_{ij}^2 = x_{ij} \in B_i \otimes B_j$$

If $x_{00} = e_0 \otimes e_0$, then $(m \circ (\text{Id} \otimes \epsilon) \circ \Delta)(e_k) \neq e_k$, thus $x_{00} = 0$.  This then implies that $\Delta(B_k) \subseteq kG \otimes (B_1 + \cdots + B_r) + (B_1 + \cdots + B_r) \otimes kG$, so that $B_1 + \cdots + B_r$ is a Hopf ideal of $kG$.  Thus there is a normal subgroup scheme $N$ of $G$ for which the kernel of the canonical map $kG \rightarrow k(G/N)$ is $B_1 + \cdots + B_r$.  This clearly implies that $k(G/N) \cong B_0$ as an algebra, from which we deduce that $G/N$ is unipotent.

On the other hand, the kernel of the projection $kG \rightarrow k(G/N)$ is $kGkN^{\dagger}$, where $kN^{\dagger}$ denotes the augmentation ideal of $kN$.  It follows that $kN^{\dagger} \subseteq B_1 + \cdots + B_r$.  Thus the image of $kN$ in $B_0$ under the composite of maps $kN \hookrightarrow kG \twoheadrightarrow B_0$ is one-dimensional.  We then have that $B_0$, as a left-module over $kN$ in the usual way (that is, as the restriction to $kN$ of a summand of the left-regular representation of $kG$) is both projective and isomorphic to $dim_k(B_0)$ copies of the trivial module.  Thus the trivial module is projective as a $kN$-module, which implies that $kN$ is semisimple.

\end{proof}

\begin{remark}
It follows from the results of \cite[1.1]{Far3} that $N$ must in fact be equal to the subgroup $G_{\text{lr}}$, which is defined to be the unique largest linearly reductive normal subgroup scheme of $G$.
\end{remark}

\section{$p$-points and Block Invariants}

In the last section we made some observations about the relationship between $V_G(\mathcal{B})$ and $X_B$, the latter space defined intrinsically for $B$.  Similarly, in this section we will look at a space defined intrinsically for $B$ and compare it with $P(G)_{\mathcal{B}}$.  The following definition is a slight modification above that given in \cite{Far1}.

\begin{defn}
Let $B$ be a block of a finite group scheme $G$.  A \textit{flat-point} of $B$ is an algebra map
$$\alpha: k[t]/(t^p) \rightarrow B$$
such that $\alpha$ is left-flat (i.e. $\alpha^*(B)$ is a projective module).  Two flat-points $\alpha, \beta$ are said to be equivalent if for all $B$-modules M:
$$\alpha^*(M) \; \text{proj.} \iff \beta^*(M) \; \text{proj.}$$

We then set $F(B)$ to be the set of all equivalence classes $[\alpha]$ of flat-points of $B$ such that there is a finitely generated $B$-module $M$ with $\alpha^*(M)$ not projective.  We also define a topology on $F(B)$ by taking the smallest topology such that the sets $\{ [\alpha] \in F(B) | \alpha^*(M) \text{ is not projective} \}$ are closed, for all $M \in B$-mod.  
\end{defn}

\begin{remark}
It is clear that $F(B)$ is defined intrinsically for $B$.  By contrast, $P(G)_{\mathcal{B}}$ is not an invariant belonging to $B$ since the definition of a $p$-point involves the Hopf algebra structure of $kG$ (due to the factorization through unipotent abelian subgroup schemes).
\end{remark}

\begin{remark}
The definition in \cite{Far1} is stated for an arbitrary algebra $A$, and the space (denoted there as $Fl(A)$) is given as the set of \textit{all} equivalence classes of flat maps from $k[t]/(t^p)$ to $A$.  We have chosen our definition in such a way that a block which is simple as an algebra will always have $F(B) = \emptyset$ (for example, if $B$ is the block of the Steinberg module for $u(\mathfrak{sl}_2)$, then $Fl(B) = \{pt.\}$).  Also, the definition in \cite{Far1} does not specify a topology on the set of flat-points.
\end{remark}

If $B$ is a block of $kG$, then the projection map $\rho: kG \rightarrow B$ is flat, and thus we get a map from $p$-points of $G$ to flat-points of $B$ by composing with $\rho$.  Also, if $M$ is a $kG$-module which lies in $B$, then $M$ is a $B$-module whose structure as a $kG$-module is given by the pull-back functor $\rho^*$.  Thus if $\alpha$ is a $p$-point of $G$, then $\rho \circ \alpha$ is a flat-point of $B$, and the $k[t]/(t^p)$-modules $\alpha^*(M)$ and $(\rho \circ \alpha)^*(M)$ are the same, since the definition of $M$ as a $kG$-module effectively involved pull-back by $\rho$ in the first place.  We see then that the map $\rho_*: P(G) \rightarrow F(B)$ is well-defined on equivalence classes of $p$-points, as the equivalence relation for flat-points is defined in the exact same manner as for $p$-points, but only considers modules lying in $B$ rather than all $kG$-modules.  To distinguish between equivalence classes in $P(G)$ and $F(B)$, we will write $[\alpha]_G$ and $[\beta]_B$ respectively.

Restricting the above map to the support space of the block, we get a map $P(G)_{\mathcal{B}} \rightarrow F(B)$.  We now show that this map is both injective and continuous.

\begin{prop}\label{injective}
Let $B$ be a block of a finite group scheme $kG$, and let $\rho$ denote the projection $kG \rightarrow B$.  The map $\rho_*: P(G)_{\mathcal{B}} \rightarrow F(B)$, which sends $[\alpha]_G$ to $[\rho \circ \alpha]_B$, is injective and continuous.
\end{prop}

\begin{proof}
Suppose that $[\alpha]_G, [\beta]_G \in P(G)_{\mathcal{B}}$, with $[\alpha]_G \ne [\beta]_G$.  By \cite[5.1]{FP1}, the inequivalence of $\alpha$ and $\beta$ as $p$-points implies that there exists a cohomology class $\zeta$ in degree $2n$, for some $n$, such that the Carlson module $L_{\zeta}$ restricts via $\alpha$ to a non-projective $k[t]/(t^p)$-module, while the restriction via $\beta$ is projective.  In other words, $[\alpha]_G \in P(G)_{L_{\zeta}}$ and $[\beta]_G \not \in P(G)_{L_{\zeta}}$.

Let $M = \bigoplus S_i$, for all simple modules $S_i$ lying in $B$, so that $P(G)_{\mathcal{B}} = P(G)_M$.  By \cite[5.6]{FP1}, we have

\begin{equation}
\label{*}
[\alpha]_G \in P(G)_{L_{\zeta} \otimes M}, \; [\beta]_G \not \in P(G)_{L_{\zeta} \otimes M}
\end{equation}

Let $e$ be the central idempotent of $B$.  The module $L_{\zeta} \otimes M$ can be decomposed into the sum of $e.(L_{\zeta} \otimes M)$ and $(1-e).(L_{\zeta} \otimes M)$.  As shown in the proof of \cite[2.2]{Far2}, since $M$ is a $B$-module, $(1-e).(L_{\zeta} \otimes M)$ is projective, hence all of the support of $L_{\zeta} \otimes M$ comes from $e.(L_{\zeta} \otimes M)$.  With this observation, (\ref{*}) then can be restated as:

$$[\alpha]_G \in P(G)_{e.(L_{\zeta} \otimes M)}, \; [\beta]_G \not \in P(G)_{e.(L_{\zeta} \otimes M)}$$

Since $e.(L_{\zeta} \otimes M)$ lies in $B$, then as observed above the action of $kG$ factors through $B$ so that $\alpha^*(e.(L_{\zeta} \otimes M)) = (\rho \circ \alpha)^*(e.(L_{\zeta} \otimes M))$, and so we have $[\rho \circ \alpha]_B \ne [\rho \circ \beta]_B$. 

\bigskip
As for the continuity of $\rho_*$, for a $B$-module $M$ we see that $(\rho_*)^{-1}(F(B)_M) = P(G)_M$, which is a closed set in $P(G)_{\mathcal{B}}$.  Since the topology on $F(B)$ is specified as being the smallest such that the $F(B)_M$ are closed sets, the continuity of $\rho_*$ follows.
\end{proof}

We will show that, in analogy with Theorem \ref{nilpotents}(1), for principal blocks having a single simple module, the map $\rho_*: P(G)_{\mathcal{B}_0} \rightarrow F(B_0)$ is a homeomorphism.  We first recall the following lemma which will be useful in the next theorem.

\begin{lemma}\label{kernel} (\cite{FP1}, \cite{Far1})
Let $\alpha$ be any flat map from $k[t]/(t^p) \rightarrow kG$.  Then the image of the induced map $\alpha^{\bullet}: \textup{H}^{\bullet}(G,k) \rightarrow \textup{H}^{\bullet}(k[t]/(t^p),k)$ is not contained in $\textup{H}^0(k[t]/(t^p),k)$.
\end{lemma}

We will now prove that any flat map to a unipotent finite group scheme is equivalent to one which factors through an abelian subgroup scheme.  Let us first though observe that such a statement is not completely trivial, by showing that there do exist flat maps to unipotent finite group schemes which do not factor through abelian subgroup schemes.

\begin{ex}
Let $G$ be a non-abelian $p$-group, and choose $g_1 \in Z(G)$ such that $|g_1| = p$.  Set

$$x= 1-g_1, \;\; N = \sum_{g \in G}g$$

The element $x+N$ is p-nilpotent, and is not contained in any subgroup algebra.  We know that the map sending $t$ to $x$ makes $kG$ into a free $k[t]/(t^p)$-module.  Since $xN = 0$, there is some element $y \in kG$ such that $N = x^{p-1}y$.  It then follows by \cite[2.2]{FP1} that the map sending $t$ to $x+N = x+x(x^{p-2}y)$ also determines a left-flat map from $k[t]/(t^p)$ to $kG$, and thus is a flat map which is not a $p$-point.
\end{ex}

\begin{thm} \label{equiv}
If $G$ is a unipotent finite group scheme, then every flat map $k[t]/(t^p) \rightarrow kG$ is equivalent to a $p$-point of $G$.
\end{thm}

\begin{proof}
Let $\alpha, \beta$ be any two flat maps such that the maps they induce in cohomology have the same kernel.  That is, $\text{ker} \; \alpha^{\bullet} = \text{ker} \; \beta^{\bullet}$.  Since $k$ is the only simple $kG$-module, then for any module $M$, we can calculate $V_G(M)$ according to the annihilator of $\text{Ext}^*_{kG}(k,M)$.  Now consider the module $k\!\Uparrow_{\alpha}^G \;:= \text{Hom}_{\alpha(k[t]/(t^p))}(kG,k)$ (i.e. the coinduced or induced module, depending on terminology, from the subalgebra of $kG$ that is the image of $k[t]/(t^p)$ under $\alpha$).  The identification of $n$-fold extensions under the Eckmann-Shapiro isomorphism (see \cite[2.8]{Ben1}), shows that the action of $H^{\bullet}(G,k)$ on $\text{Ext}^*_{kG}(k,k\!\!\Uparrow_{\alpha}^G)$ via the Yoneda product is the same as is given by $H^{\bullet}(G,k)$ acting on $\text{Ext}^*_{k[t]/(t^p)}(k,k)$ via $\alpha^{\bullet}$ followed by Yoneda product.  Thus, $\text{ker} \; \alpha^{\bullet} = \text{ker} \; \beta^{\bullet}$ implies that $V_G(k\Uparrow_{\alpha}^G) =V_G(k\Uparrow_{\beta}^G)$.  If $M$ is any $kG$-module, we then have that $V_G(k\Uparrow_{\alpha}^G \otimes M^*) =V_G(k\Uparrow_{\beta}^G \otimes M^*)$, so that in particular:

$$(k\!\Uparrow_{\alpha}^G \otimes M^*) \; \text{proj.} \; \iff (k\!\Uparrow_{\beta}^G \otimes M^*) \; \text{proj.}$$

\bigskip
\noindent We have the chain of isomorphisms:

\vspace{-0.1 in}
$$\text{Ext}^n_{kG}(k,k\!\Uparrow_{\alpha}^G \otimes M^*) \cong \text{Ext}^n_{kG}(M,k\!\Uparrow_{\alpha}^G) \cong \text{Ext}^n_{k[t]/(t^p)}(\alpha^*(M),k)$$

Since both $kG$ and $k[t]/(t^p)$ have $k$ as their only simple module, then in both module categories any non-projective module must have a non-trivial $n$-fold extension by $k$ for all $n$.  Thus

\vspace{-0.1 in}
$$\alpha^*(M) \text{ proj. } \iff (k \!\Uparrow_{\alpha}^G \otimes M^*) \text{ proj. } \iff (k \!\Uparrow_{\beta}^G \otimes M^*) \text{ proj. } \iff \beta^*(M) \text{ proj. }$$

This proves that flat maps from $k[t]/(t^p)$ to $kG$ inducing the same kernel in cohomology are equivalent.  Finally, we observe that if $\alpha$ is flat, then by Lemma \ref{kernel} we have that $\text{ker} \; \alpha^{\bullet}$ is a non-trivial maximal homogeneous ideal of $\text{H}^{\bullet}(G,k)$.  It follows from Theorem \ref{p-points} (and the definition of the map $\Psi$, see \cite[2.8]{FP1}) that there is a $p$-point $\beta$ such that $\text{ker} \; \beta^{\bullet} = \text{ker} \; \alpha^{\bullet}$, which completes the proof.
\end{proof}

\begin{remark}
This shows that in terms of giving a representation-theoretic description of support varieties, the definition of a $p$-point could simply be that it is a left-flat map factoring through a unipotent subgroup scheme.  However, the creation of $p$-points has led to new invariants for modules, and it is unclear if these invariants would work with an altered definition as suggested above.
\end{remark}

Having established the previous result, we can now prove that for principal blocks which are local, there is a homeomorphism $P(G)_{\mathcal{B}_0} \cong F(B_0)$, which, as mentioned earlier, provides a nice symmetry with Theorem \ref{nilpotents}.

\begin{prop}
Let $G$ be a finite group scheme, and suppose that the trivial module is the only simple module lying in the principal block $B_0$.  Then the map ${\rho_0}_*: P(G)_{\mathcal{B}_0} \rightarrow F(B_0)$ is a homeomorphism.
\end{prop}

\begin{proof}
Let $\alpha$ be a flat-point of $B_0$.  We have a map $\alpha^{\bullet}: \text{H}^{\bullet}(B_0,k) \rightarrow \text{H}^{\bullet}(k[t]/(t^p),k)$.  By  Proposition \ref{localunipotent}, $B_0$ is isomorphic as an algebra to the group algebra of a unipotent group scheme.   We can thus apply Lemma \ref{kernel} to see that $\text{ker} \; \alpha^{\bullet}$ is not the augmentation ideal of $\text{H}^{\bullet}(B_0,k)$.  We also have that $\rho_0$ induces an isomorphism $\rho_0^{\bullet}: \text{H}^{\bullet}(B_0,k) \stackrel{\sim}{\rightarrow} \text{H}^{\bullet}(G,k)$, thus $\rho_0^{\bullet}(\text{ker} \; \alpha_0^{\bullet})$ is not the augmentation ideal $\text{H}^{\bullet}(G,k)$.  By the same reasoning used towards the end of the proof of Theorem \ref{equiv}, there is a $p$-point $\beta$ of $kG$ such that $\beta^{\bullet}: \text{H}^{\bullet}(G,k) \rightarrow \text{H}^{\bullet}(k[t]/(t^p),k)$ satisfies

$$\text{ker} \; \alpha^{\bullet} = \text{ker} \; (\beta^{\bullet} \circ \rho_0^{\bullet})$$

Thus, $\alpha$ and ($\rho_0 \circ \beta$) are two flat-points of $B_0$ inducing the same kernel in cohomology, and since $B_0$ is isomorphic to the group algebra of a unipotent group scheme, the equivalence of $\alpha$ and $\rho_0 \circ \beta$ follows from Theorem \ref{equiv}.  Hence ${\rho_0}_*([\beta]_G) = [\alpha]_{B_0}$, and as $\alpha$ was arbitrary, ${\rho_0}_*$ is surjective.  
\end{proof}

\vspace{.2 in}
\noindent Paul Sobaje\\
Department of Mathematics \& Statistics\\
The University of Melbourne\\
Parkville, VIC, 3010\\
AUSTRALIA\\
paul.sobaje@unimelb.edu.au

\end{document}